\documentclass[]{amsart}

\usepackage[utf8]{inputenc}
\usepackage{amssymb}
\usepackage{amsthm}
\usepackage{amsmath}
\usepackage{tikz}

\usepackage[toc,page]{appendix}

\usepackage{enumitem}
\usepackage{nicefrac}

\usepackage{tikz-cd}

\usepackage{tikz}
\usetikzlibrary{datavisualization}
\usetikzlibrary{datavisualization.formats.functions}
\usetikzlibrary{calc}
\tikzset{stretch/.initial=1}
\newcommand\drawloop[4][]%
   {\draw[shorten <=0pt, shorten >=0pt,#1]
      ($(#2)!\pgfkeysvalueof{/tikz/stretch}!(#2.#3)$)
      let \p1=($(#2.center)!\pgfkeysvalueof{/tikz/stretch}!(#2.north)-(#2)$),
          \n1= {veclen(\x1,\y1)*sin(0.5*(#4-#3))/sin(0.5*(180-#4+#3))}
      in arc [start angle={#3-90}, end angle={#4+90}, radius=\n1]%
   }
\usetikzlibrary{arrows, decorations.pathreplacing, positioning,cd}
\usepackage{pgffor}

\usepackage{import}

\usepackage{mathtools}
\usepackage{mathrsfs}
\usepackage{commath}

\usepackage{framed}

\usetikzlibrary{calc}

\makeatletter
\pgfarrowsdeclare{X}{X}
{
  \pgfutil@tempdima=0.3pt%
  \advance\pgfutil@tempdima by.25\pgflinewidth%
  \pgfutil@tempdimb=5.5\pgfutil@tempdima\advance\pgfutil@tempdimb by.5\pgflinewidth%
  \pgfarrowsleftextend{+-\pgfutil@tempdimb}
  \pgfarrowsrightextend{0pt}
}
{
  \pgfutil@tempdima=0.3pt%
  \advance\pgfutil@tempdima by.25\pgflinewidth%
  \pgfsetdash{}{+0pt}
  \pgfsetroundcap
  \pgfsetmiterjoin
  \pgfpathmoveto{\pgfqpoint{-5.5\pgfutil@tempdima}{-6\pgfutil@tempdima}}
  \pgfpathlineto{\pgfqpoint{5.5\pgfutil@tempdima}{6\pgfutil@tempdima}}
  \pgfpathmoveto{\pgfqpoint{-5.5\pgfutil@tempdima}{6\pgfutil@tempdima}}
  \pgfpathlineto{\pgfqpoint{5.5\pgfutil@tempdima}{-6\pgfutil@tempdima}}
  \pgfusepathqstroke 
}
\makeatother

\newtheorem{theorem}{Theorem}[section]
\newtheorem{lemma}[theorem]{Lemma}
\newtheorem{proposition}[theorem]{Proposition}
\newtheorem{corollary}[theorem]{Corollary}
\newtheorem{example}[theorem]{Example}
\theoremstyle{definition}
\newtheorem{definition}[theorem]{Definition}

\theoremstyle{remark}
\newtheorem{remark}[theorem]{Remark}
\newtheorem{question}[theorem]{Question}

\usepackage{relsize}
\usepackage{exscale}

\usepackage{mathtools}
\usepackage{mathrsfs}

\usepackage{framed}

\usepackage{datetime}

\usepackage[pagewise]{lineno}
\usepackage{hyperref} 

\renewcommand{\epsilon}{\varepsilon}

\title[On full trajectories, limit sets and internal chain transitivity]{When is the beginning the end? On full trajectories, limit sets and internal chain transitivity} 
\author[Joel Mitchell]{Joel Mitchell}

\begin{document}

\hypersetup{pageanchor=false} 
\maketitle



\begin{abstract}
Let $f \colon X \to X$ be a continuous map on a compact metric space $X$ and let $\alpha_f$, $\omega_f$ and $ICT_f$ denote the set of $\alpha$-limit sets, $\omega$-limit sets and nonempty closed internally chain transitive sets respectively. In this paper we characterise, by introducing novel variants of shadowing, maps for which every element of $ICT_f$ is equal to (resp.\ may be approximated by) the $\alpha$-limit set and the $\omega$-limit set of the same full trajectory. We construct examples highlighting the difference between these properties.
\end{abstract}



\hypersetup{pageanchor=true} 

\section{Introduction}
Let $(X,d)$ be a compact metric space and $f \colon X \to X$ a continuous function. We call the pair $(X,f)$ a \emph{dynamical system}. Given a point $x \in X$, its \emph{orbit} is the set $\{f^i(x) \mid i \in \mathbb{N}_0\}$. The orbit sequence $x, f(x), f^2(x) \ldots$ may be thought of as where $x$ travels under iteration of $f$. In a compact metric setting, this sequence has accumulation points: the set of all such points is the $\omega$-limit set of $x$ (denoted $\omega_f(x)$). This may be thought of as the \emph{target} of $x$; it is where it \emph{ends up}, so to speak. Analogously, one may wonder where a point came from. In the case where $f$ is a homeomorphism, we may simply consider the orbit of $x$ under $f^{(-1)}$ and $\omega_{f^{(-1)}}(x)$. In this case, we call $\omega_{f^{(-1)}}(x)$ the $\alpha$-limit set of $x$ under $f$ (denoted $\alpha_f(x)$): this may be thought of as the \emph{source} of $x$. In the case when $f$ is not one-to-one, a point may have multiple sources. This presents a difficulty when attempting to give a suitable definition for $\alpha_f(x)$. Various approaches to this have been taken \cite{BalibreaPiotr, Coven,CuiDing,Hero, Sun2, Sun}. For a discussion on this, we refer the reader to \cite{GoodMeddaughMitchell}. In the present paper, we refrain from defining such sets for individual points, choosing instead to define them for \emph{backward} and \emph{full trajectories}. This is the approach taken in \cite{BalibreaPiotr}, \cite{GoodMeddaughMitchell} and \cite{Hirsch}. An infinite sequence $\langle x_i \rangle _{i \leq 0}$ is called a \emph{backward trajectory} of $x$ if $f(x_i)=x_{i+1}$ for all $i \leq -1$ and $x_0=x$. The $\alpha$-limit set of $\langle x_i \rangle _{i \leq 0}$ is the set of accumulation points of this sequence. We denote the set of all such $\alpha$-limit sets by $\alpha_f$. A \emph{full trajectory} is a two-sided sequence $\langle x_i \rangle _{i \in \mathbb{Z}}$ such that $f(x_i)=x_{i+1}$ for all $i \in \mathbb{Z}$. We define the limit sets of such a sequence in the natural way: $\alpha_f(\langle x_i \rangle _{i \in \mathbb{Z}})= \alpha_f(\langle x_i \rangle _{i \leq 0})$ and $\omega_f(\langle x_i \rangle _{i \in \mathbb{Z}})=\omega_f( x_0 )$. Although $\alpha$-limit sets have not been studied quite as extensively as their $\omega$ counterparts, interest in them has been growing (see, for example, \cite{BalibreaPiotr, Coven,CuiDing, GoodMeddaughMitchell, Hero, Hirsch}).

In this paper, we are concerned with the following two questions:
\begin{question}\label{Q1} When is it the case that every nonempty closed internally chain transitive set is both the $\alpha$-limit set and $\omega$-limit set of the same full trajectory? I.e.\ when is it true that for any $A \in ICT_f$ there exists a full trajectory $\langle x_i \rangle_{i \in \mathbb{Z}}$ such that $\alpha(\langle x_i \rangle_{i \in \mathbb{Z}})=\omega(\langle x_i \rangle_{i \in \mathbb{Z}})=A$?
\end{question}
\begin{question}\label{Q2} When is it the case that every nonempty closed internally chain transitive set may be approximated, to any given accuracy, by both the $\alpha$-limit set and $\omega$-limit set of the same full trajectory? I.e.\ when is it true that for any $A \in ICT_f$ and any $\epsilon>0$ there exists a full trajectory $\langle x_i \rangle_{i \in \mathbb{Z}}$ such that $d_H(\alpha(\langle x_i \rangle_{i \in \mathbb{Z}}), A) < \epsilon$ and $d_H( \omega(\langle x_i \rangle_{i \in \mathbb{Z}}),A)< \epsilon$?
\end{question}
In answering these questions, this paper aims to provide the final chapter in the journey to characterise when limit sets approximate, or are precisely, the elements of $ICT_f$ in terms of shadowing properties. This journey was embarked upon by Good and Meddaugh \cite{GoodMeddaugh2016} who were concerned with $\omega$-limit sets, before this author joined their path in \cite{GoodMeddaughMitchell} where we dealt mainly with $\alpha$-limit sets. Whilst this journey is, therefore, a recent one, its motivations are far older. Indeed, multiple authors have either studied, or attempted to characterise, the set of all $\omega$-limit sets in a variety of settings. For example, $\omega$-limit sets of continuous maps of the closed unit interval $I$ have been completely characterised in \cite{AgronskyBrucknerCederPearson,BrucknerSmital}: the authors show that a nonempty subset $E$ of $I$ is an $\omega$-limit set of some continuous map $f$ if and only if $E$ is either a closed, nowhere dense set, or a union of finitely many non-degenerate closed intervals. Furthermore, it has been shown that $\omega_f$ is closed (with respect to the Hausdorff topology) for maps of the circle \cite{Pokluda}, the interval \cite{BlokhBrucknerHumkeSmital} and other finite graphs \cite{MaiShao}. Perhaps one of the most important results motivating our work is one of Hirsch \emph{et al.\ }\cite{Hirsch}: every $\alpha$- and $\omega$- limit set is \emph{internally chain transitive} (precise definitions below). Together with the fact that these limit sets are closed, this means $\alpha_f, \omega_f \subseteq ICT_f$.

A second important result motivating this journey is one of Meddaugh and Raines \cite{MeddaughRaines} who establish that, for maps with \emph{shadowing}, or \emph{pseudo-orbit tracing}, $\overline{\omega_f}=ICT_f$. The shadowing property, defined below, has both numerical and theoretical importance in topological dynamics. It has been studied in a variety of settings, including, for example, in the context of Axiom A diffeomorphisms \cite{bowen-markov-partitions}, in numerical analysis \cite{Corless,CorlessPilyugin,Pearson}, as an important factor in stability theory \cite{Pilyugin, robinson-stability,Walters} and as a property in and of itself \cite{Coven1, GoodMeddaugh2018, GoodMitchellThomas, LeeSakai, Mitchell, Nusse, Pennings, Pilyugin,Sakai2003}. Various variants on the pseudo-orbit tracing property have also been studied including, for example, ergodic, thick, and Ramsey shadowing \cite{brian-oprocha, bmr, Dastjerdi, Fakhari, Oprocha2016}, limit, or asymptotic, shadowing \cite{BarwellGoodOprocha, GoodOprochaPuljiz2019, Pilyugin2007}, $s$-limit shadowing \cite{BarwellGoodOprocha,GoodOprochaPuljiz2019, LeeSakai}, orbital shadowing \cite{GoodMeddaugh2016,Mitchell, PiluginRodSakai2002, Pilyugin2007}, and inverse shadowing \cite{CorlessPilyugin, GoodMitchellThomas2, Lee}. In the first stage of this journey, Good and Meddaugh \cite{GoodMeddaugh2016} introduced new variants of shadowing which precisely characterise maps for which $\overline{\omega_f}=ICT_f$ and $\omega_f=ICT_f$. In \cite{GoodMeddaughMitchell}, the author, in collaboration with Good and Meddaugh, then characterised maps for which $\overline{\alpha_f}=ICT_f$ and $\alpha_f=ICT_f$. Along the way, we demonstrated that shadowing is itself a sufficient condition for the property under consideration in Question \ref{Q2}, whilst the addition of expansivity is sufficient for the property under consideration in Question \ref{Q1}. Due to their lengthy statements, we will refer to the properties in questions \ref{Q1} and \ref{Q2} as $P_e$ and $P_a$ respectively (`$e$' for `equal', `$a$' for `approximate'). Thus:
\begin{itemize}
    \item Property $P_e$: 
`For any $A \in ICT_f$ there exists a full trajectory $\langle x_i \rangle_{i \in \mathbb{Z}}$ such that $\alpha(\langle x_i \rangle_{i \in \mathbb{Z}})=\omega(\langle x_i \rangle_{i \in \mathbb{Z}})=A$.'

\item Property $P_a$: `For any $A \in ICT_f$ and any $\epsilon>0$ there exists a full trajectory $\langle x_i \rangle_{i \in \mathbb{Z}}$ such that $d_H(\alpha(\langle x_i \rangle_{i \in \mathbb{Z}}), A) < \epsilon$ and $d_H( \omega(\langle x_i \rangle_{i \in \mathbb{Z}}),A)< \epsilon$.'
\end{itemize}

\medskip

The layout of this paper is as follows. In Section \ref{Section_Full_Preliminaries}, we provide the definitions and motivating results which underpin this paper. In Section \ref{Section_characterise_P_e_and_P_a}, we answer questions \ref{Q1} and \ref{Q2}. Throughout, we present examples which serve both to motivate our results and also to demonstrate the distinction between some of the properties introduced in \cite{GoodMeddaugh2016}, \cite{GoodMeddaughMitchell} and in this paper. In particular, we construct an example of a system (Example \ref{Example_SquareSystem}) which demonstrates that one can have $\alpha_f=\omega_f= ICT_f$ whilst exhibiting neither property $P_e$ nor property $P_a$. Along the way, in subsection \ref{SectionGamma_Implications}, we offer a couple of implications concerning $\gamma$-limit sets (defined in said section). We close, in Section \ref{Section_closing_examples}, with two final examples, one of which (Example \ref{Example_periodic_cofinal}) is a system satisfying $P_a$ but not $P_e$.
\section{Preliminaries}\label{Section_Full_Preliminaries}

A \emph{dynamical system} is a pair $(X,f)$ consisting of a compact metric space $X$ and a continuous function $f\colon X \to X$. We say the \emph{positive orbit of $x$ under $f$} is the set of points $\{x, f(x), f^2(x), \ldots\}$. A \emph{backward trajectory} of the point $x$ is a sequence $\langle x_{i}\rangle_{i\leq0}$ for which $f(x_{i})=x_{i+1}$ for all $i\leq -1$ and $x_0=x$. We say a bi-infinite sequence $\langle x_i \rangle _{i \in \mathbb{Z}}$ is a \emph{full trajectory} (of each $x_i$) if $f(x_i)=x_{i+1}$ for each $i \in \mathbb{Z}$. Because a point may have multiple preimages, this means that a full trajectory of a point need not be unique. 

For a sequence $\langle x_i \rangle _{i>N}$ in $X$, where $N \geq -\infty$, we define its $\omega$-limit set, denoted $\omega(\langle x_i\rangle_{i > N} )$, or simply $\omega(\langle x_i\rangle)$: 
\[\omega(\langle x_i\rangle) \coloneqq \bigcap_{M \in \mathbb{N}}\overline{\{x_n \mid n >M\}}.\]
For $x \in X$, we define the $\omega$\emph{-limit set of }$x$: $\omega(x) \coloneqq \omega(\langle f^n(x)\rangle_{n = 0} ^\infty)$. 
In similar fashion, for a sequence $\langle x_i \rangle _{i< N}$ in $X$, where $N \leq \infty$, we define its \emph{$\alpha$-limit set}, denoted $\alpha(\langle x_i\rangle_{i < N} )$, or simply $\alpha(\langle x_i\rangle)$: 
\[\alpha(\langle x_i \rangle ) \coloneqq \bigcap_{M \in \mathbb{N}}\overline{\{x_n \mid n <-M\}}.\]
In the case when $f$ is a homeomorphism, we also define the $\alpha$-limit set of a point: $\alpha(x) \coloneqq  \alpha(\langle f^i(x)\rangle_{i \leq 0})$.
We denote by $\omega_f$ and $\alpha_f$ the set of all $\omega$-limit sets of points in $(X,f)$ and the set of all $\alpha$-limit sets of full trajectories in $(X,f)$ respectively. The compactness of $X$ guarantees that elements of $\alpha_f$ and $\omega_f$ are nonempty and closed.

A finite or infinite sequence $\langle x_i\rangle_{i=0} ^N$ is said to be an \emph{$\delta$-chain} if $d(f(x_i), x_{i+1})< \delta$ for all indices $i<N$. If $N=\infty$ then we say the sequence is a \emph{$\delta$-pseudo-orbit.} A set $A$ is \emph{internally chain transitive} if for any pair of points $a,b \in A$ and any $\delta>0$ there exists a finite $\delta$-chain $\langle x_i \rangle _{i=0} ^N$ in $A$ with $x_0=a$, $x_N=b$ and $N\geq 1$. We denote by $ICT_f$ the set of all nonempty closed internally chain transitive sets.

We denote by $2^X$ the hyperspace of nonempty compact subsets of $X$. This forms a compact metric space with the \emph{Hausdorff metric} induced by the metric $d$. For $A,B \in 2^X$ the \emph{Hausdorff distance} between $A$ and $B$ is given by
\[d_H (A,A^\prime)= \inf \{\epsilon>0 \mid A \subseteq B_\epsilon (A^\prime) \text{ and } A^\prime \subseteq B_\epsilon (A)\}. \]
As collections of nonempty compact sets, $\alpha_f$, $\omega_f$ and $ICT_f$ are all subsets of $2^X$.
Meddaugh and Raines \cite{MeddaughRaines} establish the following result.

\begin{lemma}\textup{\cite{MeddaughRaines}}\label{lemmaICTclosed}
Let $(X,f)$ be a dynamical system. Then $ICT_f$ is closed in $2^X$.
\end{lemma}
Hirsch \emph{et al.}\ \cite{Hirsch} show that the $\alpha$-limit set (resp.\ $\omega$-limit set) of any pre-compact backward (resp.\ forward) trajectory is internally chain transitive. Since our setting is a compact metric space all $\alpha$-limit sets and $\omega$-limit sets are internally chain transitive. We formulate this as Lemma \ref{LemmaAlphaOmegaAreICT} below.
\begin{lemma}\textup{\cite{Hirsch}} \label{LemmaAlphaOmegaAreICT} Let $(X,f)$ be a dynamical system. Then $\alpha_f, \omega_f \subseteq ICT_f$.
\end{lemma}

A point $x$ is said to \emph{$\epsilon$-shadow} a sequence $\langle x_i\rangle_{i=0} ^\infty$ if $d(f^i(x), x_i)< \epsilon$ for all $i \in \mathbb{N}_0$. We say the system $(X,f)$ has the \emph{shadowing property}, or simply shadowing, if for every $\epsilon>0$ there exists $\delta>0$ such that every $\delta$-pseudo-orbit is $\epsilon$-shadowed.

\begin{definition}
Suppose that $(X,f)$ is a dynamical system. A sequence $\langle x_{i}\rangle_{i=0} ^\infty$ in $X$ is called an \emph{asymptotic pseudo-orbit} if $d(f(x_i), x_{i+1}) \rightarrow 0$ as $i \rightarrow \infty$. In similar fashion, a sequence $\langle x_{i}\rangle_{i\leq0}$ is
\begin{enumerate}
\item a \emph{backward $\delta$-pseudo-orbit} if $d(f(x_{i}),x_{i+1})<\delta$ for each $i\leq-1$;
\item a \emph{backward asymptotic pseudo-orbit} if $d(f(x_{i}),x_{i+1})\rightarrow 0$ as $i \rightarrow -\infty$.
\end{enumerate}
A sequence $\langle x_{i}\rangle_{i\in\mathbb Z}$ in $X$ is
\begin{enumerate}
    \item a \emph{two-sided $\delta$-pseudo-orbit} if $d(f(x_{i}),x_{i+1})<\delta$ for each $i\in\mathbb Z$;
    \item a \emph{two-sided asymptotic pseudo-orbit} if $d(f(x_i), x_{i+1}) \rightarrow 0$ as $i \rightarrow \pm \infty$.
\end{enumerate}
\end{definition}
With the above terminology, the system $(X,f)$ has \emph{backward shadowing} if for any $\epsilon>0$ there exists $\delta>0$ such that for any backward $\delta$-pseudo-orbit $\langle x_i \rangle _{i \leq 0}$ there exists a backward trajectory $\langle z_i \rangle _{i \leq 0}$ such that $d(x_i, z_i) < \epsilon$ for all $i \leq 0$. Similarly it has \emph{two-sided shadowing} if for any $\epsilon>0$ there exists $\delta>0$ such that for any two-sided $\delta$-pseudo-orbit $\langle x_i \rangle _{i \in \mathbb{Z}}$ there exists a full trajectory $\langle z_i \rangle _{i \in \mathbb{Z}}$ such that $d(x_i, z_i) < \epsilon$ for all $i \in \mathbb{Z}$. In \cite{GoodMaciasMeddaughMitchellThomas}, we show that shadowing implies backward and two-sided shadowing, whilst all three properties are equivalent for surjective maps. Using this, in \cite{GoodMeddaughMitchell}, we showed that shadowing is sufficient for $P_a$. Furthermore, we showed that the addition of \emph{expansivity} is sufficient for $P_e$; a map is \emph{expansive} if there exists $c>0$ such that for any distinct $x,y \in X$ there exists $k \in \mathbb{N}_0$ such that $d(f^k(x), f^k(y)) \geq c$. These results are formulated below.

\begin{theorem}\cite[Theorem 4.2]{GoodMeddaughMitchell}\label{thm_shad_implies_P_a} If $(X,f)$ has shadowing then it satisfies property $P_a$.
\end{theorem}
\begin{theorem}\cite[Theorem 4.10]{GoodMeddaughMitchell}\label{thm_shad_plus_expansive_implies_P_e}
If $(X,f)$ is an expansive system with shadowing then it satisfies property $P_e$.
\end{theorem}
Embedded within the proof of \cite[Theorem 4.2]{GoodMeddaughMitchell} is the following result. It will be important for our characterisations of $P_a$ and $P_e$.

\begin{lemma}\cite{GoodMeddaughMitchell} \label{Lemma_KEYLEMMA_ICT}
Let $(X,f)$ be a dynamical system. For any $A \in ICT_f$ and any $\epsilon>0$ there exists a two-sided asymptotic $\epsilon$-pseudo-orbit $\langle a_i \rangle_{i \in \mathbb{Z}}$ in $A$ such that $\alpha(\langle a_i \rangle)=\omega(\langle a_i \rangle)=A$.
\end{lemma}

The system $(X,f)$ has \emph{limit shadowing}, also called asymptotic shadowing, if every asymptotic pseudo-orbit $\langle x_i \rangle_{i=0} ^\infty$ is \emph{asymptotically shadowed} (i.e.\ there exists $z \in X$ for which $d(f^i(z), x_i) \to 0$ as $i \to \infty$). It has \emph{backward limit shadowing} if every backward asymptotic pseudo-orbit $\langle x_i \rangle_{i\leq 0} $ is backward asymptotically shadowed (i.e.\ there exists a backward trajectory $\langle z_i \rangle_{i \leq 0}$ for which $d(z_i, x_i) \to 0$ as $i \to -\infty$). Finally the system has \emph{two-sided limit shadowing} if every two-sided asymptotic pseudo-orbit $\langle x_i \rangle_{i\in \mathbb{Z}}$ is two-sided asymptotically shadowed (i.e.\ there exists a full trajectory $\langle z_i \rangle_{i \in \mathbb{Z}}$ for which $d(z_i, x_i) \to 0$ as $i \to \pm \infty$). The following proposition combines results from \cite{BarwellGoodOprochaRaines} and \cite{GoodMeddaughMitchell}.

\begin{proposition}\cite{BarwellGoodOprochaRaines,GoodMeddaughMitchell}
If $(X,f)$ has limit (resp.\ backward limit) shadowing then $\omega_f=ICT_f$ (resp.\ $\alpha_f=ICT_f$).
\end{proposition}

\section{Characterising properties $P_e$ and $P_a$}\label{Section_characterise_P_e_and_P_a}

\subsection{Property $P_e$}

Recall property $P_e$: 

`For any $A \in ICT_f$ there exists a full trajectory $\langle x_i \rangle_{i \in \mathbb{Z}}$ such that $\alpha(\langle x_i \rangle)=\omega(\langle x_i \rangle)=A$.'

It is obvious, given Lemma \ref{LemmaAlphaOmegaAreICT}, that a necessary condition for $P_e$ is $\alpha_f=\omega_f=ICT_f$. A natural starting point, therefore, is to ask if this is also sufficient. It turns out that this is not the case: in Example \ref{Example_SquareSystem} we construct a homeomorphism for which $\alpha_f=\omega_f=ICT_f$ but for which $P_e$ is not satisfied.

\begin{example}\label{Example_SquareSystem}
We will build up the points in $X$ and define the map $f \colon X \to X$ on them as we go. Let $X$ consist of the following points in the Cartesian plane. Let $(0,0)$ be a fixed point, so that $f(0,0)=(0,0)$. Let $(1,1), (-1,1), (-1,-1)$ and $(1,-1)$ also be fixed points and consider the following subsets of the $2 \times 2$ square $S$ with these four points as vertices:
\[A = \left\{ \left(\pm \frac{2^n- 1}{2^n},1\right) \mid n \in \mathbb{N}_0\right\},\]
\[B = \left\{ \left(-1, \pm \frac{2^n- 1}{2^n}\right)  \mid n \in \mathbb{N}_0\right\},\]
\[C = \left\{ \left(\pm\frac{2^n- 1}{2^n},-1\right) \mid n \in \mathbb{N}_0\right\},\]
\[D = \left\{ \left(1,\pm\frac{2^n- 1}{2^n}\right) \mid n \in \mathbb{N}_0\right\}.\]
We now define the map $f$ on these four sets so that points move anticlockwise between the two vertices which are the limit points of said set. 

For example, for $A$ the vertices which form the limit points of $A$ are $(1,1)$ and $(-1,1)$. For any $z \in A$, with $z= (\frac{2^n- 1}{2^n},1)$ for some $n \geq 1$, let $f(z)= (\frac{2^{n-1}- 1}{2^{n-1}},1)$. Let $f(0,1)=(-\frac{1}{2},1)$. Finally, for any $z \in A$, with $z= (-\frac{2^n- 1}{2^n},1)$ for some $n \geq 1$, let $f(z)= (-\frac{2^{n+1}- 1}{2^{n+1}},1)$. Define $f$ on $B,C$ and $D$ similarly, with this anticlockwise movement. We let $Q= \{(1,1), (-1,1), (-1,-1), (1,-1)\} \cup A \cup B \cup C \cup D$.

Next, we insert the points given by $(0,\frac{1}{2^n})$ for each $n \geq 2$. For each of these, we let $f(0,\frac{1}{2^n})= (0,\frac{1}{2^{n-1}})$.

We now, for each $n \in \mathbb{N}$ insert a finite subset of a square as follows: Insert the points $(\frac{2^n- 1}{2^n},\frac{2^n- 1}{2^n}), (-\frac{2^n- 1}{2^n},\frac{2^n- 1}{2^n}),(-\frac{2^n- 1}{2^n},-\frac{2^n- 1}{2^n})$ and $(\frac{2^n- 1}{2^n},-\frac{2^n- 1}{2^n})$ (these are the vertices). For each $n \in \mathbb{N}$ we insert the following finite subsets of these squares:
\[A_n = \left\{ \left(\pm \frac{2^m- 1}{2^m},\frac{2^n- 1}{2^n}\right) \mid m \in \mathbb{N}_0 \text{ and } \lvert \frac{2^m- 1}{2^m} \rvert < \frac{2^n- 1}{2^n} \right\},\]

\[B_n = \left\{ \left(-\frac{2^n- 1}{2^n}, \pm \frac{2^m- 1}{2^m}\right) \mid m \in \mathbb{N}_0 \text{ and } \lvert \frac{2^m- 1}{2^m} \rvert < \frac{2^n- 1}{2^n} \right\},\]

\[C_n = \left\{ \left(\pm \frac{2^m- 1}{2^m},-\frac{2^n- 1}{2^n}\right) \mid m \in \mathbb{N}_0 \text{ and } \lvert \frac{2^m- 1}{2^m} \rvert < \frac{2^n- 1}{2^n} \right\},\]

\[D_n = \left\{ \left(\frac{2^n- 1}{2^n}, \pm \frac{2^m- 1}{2^m}\right) \mid m \in \mathbb{N}_0 \text{ and } \lvert \frac{2^m- 1}{2^m} \rvert < \frac{2^n- 1}{2^n} \right\}.\]

Let $Q_n=\{(\frac{2^n- 1}{2^n},\frac{2^n- 1}{2^n}), (-\frac{2^n- 1}{2^n},\frac{2^n- 1}{2^n}),(-\frac{2^n- 1}{2^n},-\frac{2^n- 1}{2^n}), (\frac{2^n- 1}{2^n},-\frac{2^n- 1}{2^n})\} \cup A_n \cup B_n \cup C_n \cup D_n$.

For each $n \in \mathbb{N}$, let $f(\frac{2^n- 1}{2^n},\frac{2^n- 1}{2^n})=(\frac{2^n- 1}{2^n},\frac{2^{n+1}- 1}{2^{n+1}})$. All points in $Q_n \setminus \{ (\frac{2^n- 1}{2^n},\frac{2^n- 1}{2^n})\}$, as before, move anticlockwise around the finite set $Q_n$ under $f$. So that, in $Q_1$ for example, $f(0, \frac{1}{2})=(-\frac{1}{2}, \frac{1}{2}), f(-\frac{1}{2}, \frac{1}{2})= (-\frac{1}{2}, 0), f(-\frac{1}{2}, 0)=(-\frac{1}{2}, -\frac{1}{2}), \ldots f(\frac{1}{2}, 0)= (\frac{1}{2}, \frac{1}{2})$. 

It follows that the $\omega$-limit set of every point, apart from $(0,0)$, which lies inside the region bounded by $Q$ in the plane is $Q$, whilst their $\alpha$-limit set is $\{(0,0)\}$.

Now input the points $(0,y) \in \mathbb{R}^2$ such that $y=\frac{3}{2}+\frac{2^n- 1}{2^{n+1}}$ for some $n \in \mathbb{N}_0$. Let $f(0,\frac{3}{2}+\frac{2^n- 1}{2^{n+1}})=(0,\frac{3}{2}+\frac{2^{n+1}- 1}{2^{n+2}})$. Let the limit this sequence, $(0, 2)$, be a fixed point under $f$.

Now, for each $n \in \mathbb{N}$ insert a finite subset of a square as follows: Insert the points $(1+\frac{1}{2^n},1+\frac{1}{2^n}), (1-\frac{1}{2^n},1+\frac{1}{2^n}),(-1-\frac{1}{2^n},-1-\frac{1}{2^n})$ and $(1+\frac{1}{2^n},-1-\frac{1}{2^n})$ (these are the vertices). For each $n \in \mathbb{N}$ we insert the following finite subsets of these squares:

\[E_n = \left\{ \left(\pm \frac{2^m- 1}{2^m},1+ \frac{1}{2^n}\right) \mid m \in \mathbb{N}_0 \text{ and } \lvert \frac{2^m- 1}{2^m} \rvert < \frac{2^n- 1}{2^n} \right\},\]

\[F_n = \left\{ \left(-1-\frac{1}{2^n}, \pm \frac{2^m- 1}{2^m}\right) \mid m \in \mathbb{N}_0 \text{ and } \lvert \frac{2^m- 1}{2^m} \rvert < \frac{2^n- 1}{2^n} \right\},\]

\[G_n = \left\{ \left(\pm \frac{2^m- 1}{2^m},-1-\frac{1}{2^n}\right) \mid m \in \mathbb{N}_0 \text{ and } \lvert \frac{2^m- 1}{2^m} \rvert < \frac{2^n- 1}{2^n} \right\},\]

\[H_n = \left\{ \left(1+\frac{1}{2^n}, \pm \frac{2^m- 1}{2^m}\right) \mid m \in \mathbb{N}_0 \text{ and } \lvert \frac{2^m- 1}{2^m} \rvert < \frac{2^n- 1}{2^n} \right\}.\]

Let $R_n=\{(1+\frac{1}{2^n},1+\frac{1}{2^n}), (-1-\frac{1}{2^n},1+\frac{1}{2^n}),(-1-\frac{1}{2^n},-1-\frac{1}{2^n}), (1+\frac{1}{2^n},-1-\frac{1}{2^n})\} \cup E_n \cup F_n \cup G_n \cup H_n$.

For each $n \geq 2$, let $f( -\frac{2^{n-1}- 1}{2^{n-1}},1+\frac{1}{2^n})=(-\frac{2^{n-1}- 1}{2^{n-1}},1+\frac{1}{2^{n-1}})$. 

For each $n \in \mathbb{N}$, let all points in $R_n \setminus \{ ( -\frac{2^{n-1}- 1}{2^{n-1}},1+\frac{1}{2^n})\}$, as before, move anticlockwise around the finite set $R_n$ under $f$. So that, in $R_1$ for example, $f(-\frac{1}{2}, \frac{3}{2})=(-\frac{3}{2}, \frac{3}{2}), f(-\frac{3}{2}, \frac{3}{2})= (-\frac{3}{2}, \frac{1}{2}), f(-\frac{3}{2}, \frac{1}{2})=(-\frac{3}{2}, 0), \ldots f(\frac{3}{2}, \frac{3}{2})= (\frac{1}{2}, \frac{3}{2}), f(\frac{1}{2},\frac{3}{2})=(0,\frac{3}{2})$. 

Then $\alpha_f= \omega_f =ICT_f$ but property $P_e$ is not satisfied.
\end{example}
It is easily observed by looking at Figure \ref{FigureSquareExample} that in Example \ref{Example_SquareSystem}, $\alpha_f$, $\omega_f$ and $ICT_f$ are all equal to \[\{ Q, \{(0,0)\},\{(1,1)\},\{(-1,1)\},\{(-1,-1)\}, \{(1,-1)\}, \{(0,2)\}\}.\]
However it is also clear that no full trajectory has $Q$ as both its $\alpha$-limit set and $\omega$-limit set.
\begin{figure}[h!] \label{FigureSquareExample}
\centering
\begin{tikzpicture}[>=stealth',scale=3.2]

	\fill[black] (0,0) circle (.2mm);
	\node[circle,draw=white] (origin) at (0,0) {};
	\drawloop[->,stretch=0.6]{origin}{225}{320};

	\fill[black] (-1,-1) circle (.2mm);
	\node[circle,draw=white] (origin) at (-1,-1) {};
	\drawloop[->,stretch=0.4]{origin}{180}{300};

	\fill[black] (1,1) circle (.2mm);
	\node[circle,draw=white] (end) at (1,1) {};
	\drawloop[<-,stretch=0.4]{end}{320}{440};;

	\fill[black] (-1,1) circle (.2mm);
	\node[circle,draw=white] (origin) at (-1,1) {};
	\drawloop[->,stretch=0.4]{origin}{100}{220};
	
	\fill[black] (1,-1) circle (.2mm);
	\node[circle,draw=white] (end) at (1,-1) {};
	\drawloop[<-,stretch=0.4]{end}{280}{400};;
	
	\foreach \y in {0,0.5,-.5,0.75,-.75,0.875,-.875,}
		{
		\fill[black] (-1,\y) circle (.2mm);
		}
	\foreach \y/\x in {0.845/0.78,0.72/.53,.47/.03,-.03/-.47,-.53/-.72,-.78/-.845}
		{
		\draw[->] (-1,\y) -- (-1,\x);
		}
		\draw[dotted, -] (-1,1) -- (-1,0.875);
		\draw[dotted, -] (-1,-.875) -- (-1,-1);

		\foreach \y in {0,0.5,-.5,0.75,-.75,0.875,-.875,}
		{
		\fill[black] (1,\y) circle (.2mm);
		}
		\foreach \y/\x in {0.845/0.78,0.72/.53,.47/.03,-.03/-.47,-.53/-.72,-.78/-.845}
		{
		\draw[->] (1,\x) -- (1,\y);
		}
		\draw[dotted, -] (1,0.875) -- (1,1);
		\draw[dotted, -] (1,-1) -- (1,-.875);

		\foreach \x in {0,0.5,-.5,0.75,-.75,0.875,-.875,}
		{
		\fill[black] (\x,-1) circle (.2mm);
		}
		\foreach \y/\x in {0.845/0.78,0.72/.53,.47/.03,-.03/-.47,-.53/-.72,-.78/-.845}
		{
		\draw[->] (\x,-1) -- (\y,-1);
		}
		\draw[dotted, -] (0.875,-1) -- (1,-1);
		\draw[dotted, -] (-1,-1) -- (-.875,-1);

		\foreach \x in {0,0.5,-.5,0.75,-.75,0.875,-.875,}
		{
		\fill[black] (\x,1) circle (.2mm);
		}
			\foreach \y/\x in {0.845/0.78,0.72/.53,.47/.03,-.03/-.47,-.53/-.72,-.78/-.845}
		{
		\draw[->] (\y,1) -- (\x,1);
		}
		\draw[dotted, -] (1,1) -- (0.875,1);
		\draw[dotted, -] (-.875,1) -- (-1,1);
		
		\foreach \y in {0,0.5,0.25, 0.125}
		{
		\fill[black] (0,\y) circle (.2mm);
		}
			\foreach \x/\y in {0.28/0.47,0.15/.22}
		{
		\draw[->] (0,\x) -- (0,\y);
		}
		\draw[dotted, -] (0,0) -- (0,.125);
		
		\foreach \x in {0.5,-.5,0}
		{
		\fill[black] (\x,.5) circle (.2mm);
		}
		\draw[->] (-0.03,0.5) -- (-.47,0.5);
		
		\foreach \x in {0.5,-.5,0}
		{
		\fill[black] (\x,-.5) circle (.2mm);
		}
		\draw[->]  (-.47,-0.5) -- (-0.03,-0.5);
		\draw[->]  (.03,-.5) -- (.47,-0.5);
		
		\foreach \y in {0.5,-.5,0}
		{
		\fill[black] (.5,\y) circle (.2mm);
		}
		\draw[->]  (.5,-.47) -- (.5,-0.03);
		\draw[->]  (.5,.03) -- (.5,0.47);
		
		\foreach \y in {0.5,-.5, 0}
		{
		\fill[black] (-.5,\y) circle (.2mm);
		}
		\draw[->]  (-.5,.47) -- (-0.5,.03);
		\draw[->]  (-.5, -.03) -- (-0.5,-.47);
		
	\draw[->]  (.5,.53) -- (.5,.72);
		\foreach \x in {0.5,-.5,0,-.75,.75}
		{
		\fill[black] (\x,.75) circle (.2mm);
		}
	    \draw[->] (.47,0.75) -- (0.03,0.75);
		\draw[->] (-0.03,0.75) -- (-.47,0.75);
        \draw[->] (-0.53,0.75) -- (-.72,0.75);

		\foreach \x in {0.5,-.5,0,.75,-.75}
		{
		\fill[black] (\x,-.75) circle (.2mm);
		}
		 \draw[->] (-.72,-0.75) -- (-0.53,-0.75);
		 	\draw[->] (-.47,-0.75) -- (-0.03,-0.75);
		 	\draw[->]  (.03,-.75) -- (.47,-0.75);
		 	 \draw[->] (0.53,-0.75) -- (.72,-0.75);
		
		\foreach \y in {0.5,-.5,0, 0.75,-.75}
		{
		\fill[black] (.75,\y) circle (.2mm);
		}
		
		 \draw[->] (0.75,-.72) -- (.75,-0.53);
		 	\draw[->] (.75,-.47) -- (.75,-0.03);
		 	\draw[->]  (.75,.03) -- (.75,.47);
		 	 \draw[->] (.75,0.53) -- (.75,.72);
		
		\foreach \y in {0.5,-.5, 0}
		{
		\fill[black] (-.75,\y) circle (.2mm);
		}
			 \draw[->] (-.75,.72) -- (-.75, 0.53);
		 	\draw[->] (-.75,.47) -- (-.75,0.03);
		 	\draw[->]  (-.75, -.03) -- (-.75,-.47);
		 	 \draw[->] (-.75, -0.53) -- (-.75, -.72);
		 	 
		 	\draw[->]  (.75,.78) -- (.75,.845);
		 	\fill[black] (0.75,0.875) circle (.2mm);
		 		\draw[dotted, ->] (0.75,0.875) -- (-0.845,0.875);
		 		\draw[dotted, ->] (-0.875,0.845) -- (-0.875,-0.845);
		 			\draw[dotted, ->] (-0.845,-0.875) -- (0.845,-0.875);
		 			\draw[dotted, -] (0.875,-0.845) -- (0.875,0.9075);

		
	\fill[black] (0,2) circle (.2mm);
	\node[circle,draw=white] (end) at (0,2) {};
	\drawloop[<-,stretch=0.4]{end}{400}{520};;

		\foreach \y in {1.5,1.75,1.875}
		{
		\fill[black] (0,\y) circle (.2mm);
		}
			\foreach \x/\y in {1.53/1.72,1.78/1.845}
		{
		\draw[->] (0,\x) -- (0,\y);
		}
		\draw[dotted, -] (0,1.845) -- (0,2);

		\foreach \x in {0.5,-.5,0,1.5,-1.5}
		{
		\fill[black] (\x,1.5) circle (.2mm);
		}
		\draw[->] (0.47,1.5) -- (0.03,1.5);
		\draw[->] (1.47,1.5) -- (0.53,1.5);
		\draw[->] (-0.53,1.5) -- (-1.47,1.5);
		
		\foreach \x in {0.5,-.5,0,1.5,-1.5}
		{
		\fill[black] (\x,-1.5) circle (.2mm);
		}
		\draw[->]  (-1.47,-1.5) -- (-0.53,-1.5);
		\draw[->]  (-.47,-1.5) -- (-0.03,-1.5);	
		\draw[->]  (.03,-1.5) -- (.47,-1.5);
		\draw[->]  (.53,-1.5) -- (1.47,-1.5);
		
		\foreach \y in {0.5,-.5,0,1.5,-1.5}
		{
		\fill[black] (1.5,\y) circle (.2mm);
		}
		\draw[->]  (1.5,-1.47) -- (1.5,-0.53);
		\draw[->]  (1.5,-.47) -- (1.5,-0.03);	
		\draw[->]  (1.5,.03) -- (1.5,.47);
		\draw[->]  (1.5,.53) -- (1.5,1.47);
		
		\foreach \y in {0.5,-.5,0,1.5,-1.5}
		{
		\fill[black] (-1.5,\y) circle (.2mm);
		}
		\draw[->] (-1.5,1.47) -- (-1.5,0.53);
		\draw[->] (-1.5,0.47) -- (-1.5,0.03);
		\draw[->] (-1.5,-0.03) -- (-1.5,-0.47);
		\draw[->] (-1.5,-0.53) -- (-1.5,-1.47);
		
	\draw[->]  (-.5,1.28) -- (-.5,1.47);
		\foreach \x in {0.5,-.5,0,-.75,.75, 1.25, -1.25}
		{
		\fill[black] (\x,1.25) circle (.2mm);
		}
		\draw[->] (1.22,1.25) -- (.78,1.25);
		\draw[->] (.72,1.25) -- (0.53,1.25);
	    \draw[->] (.47,1.25) -- (0.03,1.25);
		\draw[->] (-0.03,1.25) -- (-.47,1.25);
        \draw[->] (-0.78,1.25) -- (-1.22,1.25);

		\foreach \x in {0.5,-.5,0,-.75,.75, 1.25, -1.25}
		{
		\fill[black] (\x,-1.25) circle (.2mm);
		}
			\draw[->] (-1.22,-1.25) -- (-.78,-1.25);
		 \draw[->] (-.72,-1.25) -- (-0.53,-1.25);
		 	\draw[->] (-.47,-1.25) -- (-0.03,-1.25);
		 	\draw[->]  (.03,-1.25) -- (.47,-1.25);
		 	 \draw[->] (0.53,-1.25) -- (.72,-1.25);
		 	 \draw[->] (0.78,-1.25) -- (1.22,-1.25);
		
		\foreach \y in {0.5,-.5,0,-.75,.75, 1.25, -1.25}
		{
		\fill[black] (1.25,\y) circle (.2mm);
		}
		 \draw[->] (1.25,-1.22) -- (1.25,-0.78);
		 \draw[->] (1.25,-.72) -- (1.25,-0.53);
		 	\draw[->] (1.25,-.47) -- (1.25,-0.03);
		 	\draw[->]  (1.25,.03) -- (1.25,.47);
		 	 \draw[->] (1.25,0.53) -- (1.25,.72);
		 	  \draw[->] (1.25,.78) -- (1.25,1.22);
		
		\foreach \y in {0.5,-.5, 0,-.75,.75, 1.25, -1.25}
		{
		\fill[black] (-1.25,\y) circle (.2mm);
		}
		 \draw[->] (-1.25,1.22) -- (-1.25, .78);
			 \draw[->] (-1.25,.72) -- (-1.25, 0.53);
		 	\draw[->] (-1.25,.47) -- (-1.25,0.03);
		 	\draw[->]  (-1.25, -.03) -- (-1.25,-.47);
		 	 \draw[->] (-1.25, -0.53) -- (-1.25,-.72);
		 	 \draw[->] (-1.25,-.78) -- (-1.25, -1.22);

		 	\draw[->]  (-.75,1.15) -- (-.75,1.22);
		 	\fill[black] (-0.75,1.125) circle (.2mm);
		 		\draw[dotted, ->] (1.095,1.125) -- (-.72,1.125);
		 		\draw[dotted, ->] (1.125,-1.095) -- (1.125,1.095);
		 			\draw[dotted, ->] (-1.095,-1.125) -- (1.095,-1.125);
		 			\draw[dotted, ->] (-1.125,1.095) -- (-1.125,-1.095);
		 			\draw[dotted, ->] (-0.905,1.125) -- (-1.095,1.125);

\end{tikzpicture}
\caption{The construction from Example \ref{Example_SquareSystem}}
\end{figure}
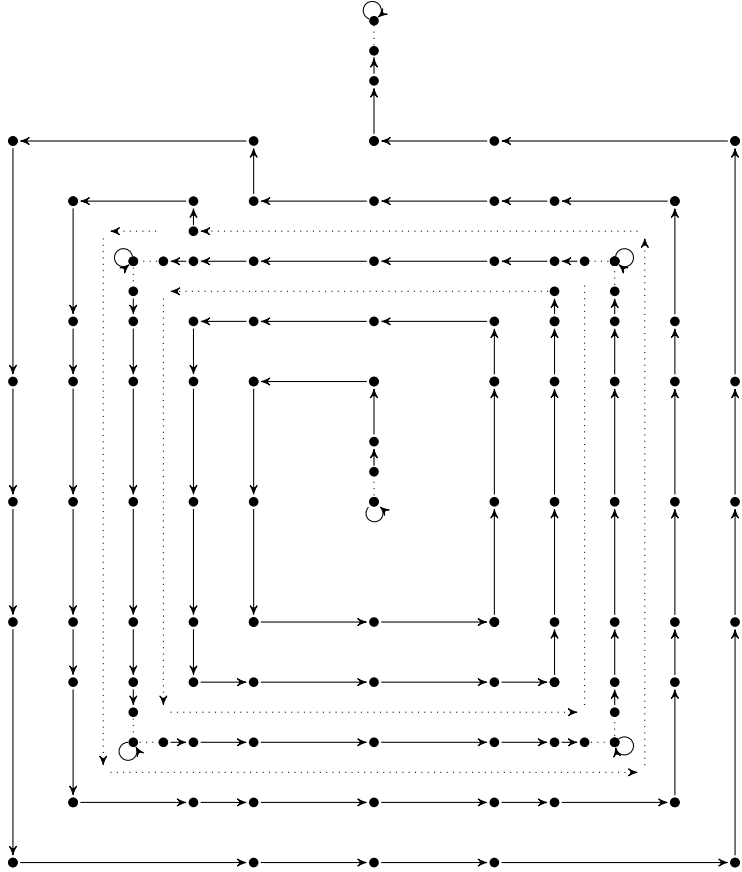

\medskip

In order to characterise $P_e$, it is will be helpful to consider the characterisations of when $\omega_f=ICT_f$ and $\alpha_f=ICT_f$. A system $(X,f)$ has the \emph{orbital limit shadowing} property, as introduced by Pilyugin \cite{Pilyugin2007}, if for any asymptotic pseudo-orbit $\langle x_i \rangle_{i\geq 0}$ there exists a point $z\in X$ such that $\omega(z) = \omega(\langle x_i \rangle)$. In \cite{GoodMeddaugh2016} the authors introduce two novel variants of shadowing, namely \emph{asymptotic orbital shadowing} and \emph{asymptotic strong orbital shadowing}. Whilst we omit the definitions here, Good and Meddaugh \cite[Theorem 22]{GoodMeddaugh2016} show that these are equivalent to orbital limit shadowing. They go on to show that (see \cite[Theorem 22]{GoodMeddaugh2016}) $\omega_f=ICT_f$ if and only if $(X,f)$ has orbital limit shadowing. Backward analogues of orbital limit shadowing, asymptotic orbital shadowing and asymptotic strong orbital shadowing, which were also shown to be equivalent to each other, were introduced in \cite{GoodMeddaughMitchell} and demonstrated to characterise when $\alpha_f=ICT_f$ (see \cite[Theorem 5.12]{GoodMeddaughMitchell}). The system $(X,f)$ exhibits \emph{backward orbital limit shadowing} (the backward analogue of orbital limit shadowing) if for each backward asymptotic pseudo-orbit $\langle x_i \rangle _{i \leq 0}$ there exists a backward trajectory $\langle z_i \rangle _{i \leq 0}$ such that $\alpha(\langle z_i \rangle)=\alpha(\langle x_i \rangle)$. It follows that the system constructed in Example \ref{Example_SquareSystem} has both orbital limit shadowing and backward orbital limit shadowing. (As an aside, we remark that there is no general entailment between orbital limit shadowing and backward orbital limit shadowing: In \cite{GoodMeddaughMitchell} we construct an example for which $\alpha_f =ICT_f$ but $\omega_f \neq ICT_f$ and another where $\omega_f =ICT_f$ but $\alpha_f \neq ICT_f$.)

These results suggest the following shadowing property.

\begin{definition}
A system $(X,f)$ has \emph{two-sided orbital limit shadowing} if for any two-sided asymptotic pseudo-orbit $\langle x_i \rangle _{i \in \mathbb{Z}}$ there exists a full trajectory $\langle z_i \rangle _{i \in \mathbb{Z}}$ such that $\alpha(\langle z_i \rangle)=\alpha(\langle x_i \rangle)$ and $\omega(\langle z_i \rangle) = \omega(\langle x_i \rangle)$.
\end{definition}

This property is strictly weaker than the two-sided limit shadowing property studied by various authors (e.g.\ \cite{Carvalho2015, Carvalho2018, CarvalhoKwietniak,GoodMaciasMeddaughMitchellThomas, Oprocha2014}). An irrational rotation of the circle will have two-sided orbital limit shadowing but not two-sided limit shadowing. We will see in Corollary \ref{corollary_sufficient_P_e} that this property is sufficient for $P_e$. However, it is too strong for our purposes. Indeed, $f \colon [0,1] \to [0,1] \colon x \mapsto x^2$ satisfies property $P_e$ but it does not have two-sided orbital limit shadowing. To see this consider the two-sided asymptotic pseudo-orbit $\langle x_i \rangle_{i \in \mathbb{Z}}$, given by $x_i =0$ for $i \leq 0$ and $x_i=1$ for $i>0$. Then $\omega(\langle x_i \rangle)= \{1\}$ and $\alpha(\langle x_i \rangle)= \{0\}$. However, the only full trajectory whose $\omega$-limit set is $\{1\}$ is given by $z_i =1$ for all $i \in \mathbb{Z}$, but $\alpha(1)=\{1\} \neq \alpha(\langle x_i \rangle)$. The strength of this shadowing property seems partly to lie in the lack of restriction in where the pseudo-orbit may `jump'. To overcome this we suggest the following weakening. 

\begin{definition}
A system $(X,f)$ has \emph{$\delta$-restricted two-sided orbital limit shadowing} if there exists $\delta>0$ such that for any two-sided asymptotic $\delta$-pseudo-orbit $\langle x_i \rangle _{i \in \mathbb{Z}}$ there exists a full trajectory $\langle z_i \rangle _{i \in \mathbb{Z}}$ such that $\alpha(\langle z_i \rangle)=\alpha(\langle x_i \rangle)$ and $\omega(\langle z_i\rangle) = \omega(\langle x_i \rangle)$.
\end{definition}

We will see (Corollary \ref{corollary_sufficient_P_e}), that $\delta$-restricted two-sided orbital limit shadowing is indeed sufficient for $P_e$, however, as Example \ref{Example_WeakLimit_not_necessary} shows, it is not necessary.

\begin{example}\label{Example_WeakLimit_not_necessary}
For each $n \in \mathbb{N}$, let $X_n$ be the circle $\mathbb{R}/\mathbb{Z} \times \{\frac{1}{n}\}$ and let $f_n\colon X_n \to X_n$ be given by $x \mapsto x + \alpha$, where $\alpha$ is some fixed irrational number. Let $X_0 = \mathbb{R}/\mathbb{Z} \times \{0\}$ and $f_0\colon X_0 \to X_0$ also be given by $x \mapsto x + \alpha$. Take $X= \bigcup_{n=0} ^\infty X_n$ and let $f \colon X \to X$ be defined by saying that, for any $n \in \mathbb{N}_0$ and any $x \in X_n$, $f(x)= f_n(x)$.
Then $(X,f)$ has property $P_e$ but not $\delta$-restricted two-sided orbital limit shadowing.
\end{example}
\begin{figure}[h!]
\centering
\begin{tikzpicture}[>=stealth',scale=3.2]
 \draw[-] (0,0.5) -- (1,0.5);
\draw[-] (0,0.25) -- (1,0.25);
\draw[-] (0,0.125) -- (1,0.125);
\draw[-] (0,0.0625) -- (1,0.0625);
\draw[-] (0,0.03125) -- (1,0.03125);
\draw[dotted, -] (0,0.015625) -- (1,0.015625);
\draw[-] (0,0) -- (1,0);
\end{tikzpicture}
\caption{Example \ref{Example_WeakLimit_not_necessary}}
\end{figure}

To see that $(X,f)$ in Example \ref{Example_WeakLimit_not_necessary} has property $P_e$ it suffices to note that it is simply composed of disjoint minimal systems. Since only one system, $(X_0, f_0)$, is a limit of other systems we get that $ICT_f=\{ X_n \mid n \in \mathbb{N}_0\}$. But, for any $n \in \mathbb{N}_0$, the orbit of each point in $X_n$ is dense in $X_n$. Property $P_e$ now follows. Now suppose that the system has $\delta$-restricted two-sided orbital limit shadowing. Let $\delta>0$ bear witness to this and let $\frac{1}{n} <\delta$. Let $x=(0,0)$ and $y=(0, \frac{1}{n})$. Then $\langle x_i \rangle_{i \in \mathbb{Z}}$ where $x_i= f^i(x)$ for $i \geq 0$ and $x_i= f^{i}(y)$ for $i < 0$ is a two-sided asymptotic $\delta$-pseudo-orbit which is not two-sided asymptotically shadowed, a contradiction.


Motivated by the backward and forward orbital limit shadowing properties, we define the following novel variant shadowing which does in fact characterise $P_e$ (Theorem \ref{thm_char_full_orbit_P_e}).

\begin{definition}
A system $(X,f)$ has \emph{$\gamma$-restricted two-sided orbital limit shadowing} if for any two-sided asymptotic pseudo-orbit $\langle x_i \rangle_{i \in \mathbb{Z}}$ such that $\alpha( \langle x_i \rangle )=\omega (\langle x_i \rangle)$ there exists a full trajectory $\langle z_i \rangle_{i \in \mathbb{Z}}$ such that 
\begin{enumerate}
    \item $\alpha (\langle z_i \rangle)=\alpha (\langle x_i \rangle )$; and, 
    \item $\omega (\langle z_i \rangle)=\omega (\langle x_i \rangle)$.
\end{enumerate}
\end{definition}

Before we prove Theorem \ref{thm_char_full_orbit_P_e} we require the following lemma.

\begin{lemma}\textup{\cite{Hirsch}}\label{LemmaOmegaAlphaAsymPseudoAreICT}
Let $(X,f)$ be a dynamical system where $X$ is a (not necessarily compact) metric space. The $\alpha$-limit set (resp.\ $\omega$-limit set) of any backward (resp.\ forward) pre-compact asymptotic pseudo-orbit is internally chain transitive. In particular, when $X$ is compact, all such limit sets are in $ICT_f$.
\end{lemma}


\begin{theorem}\label{thm_char_full_orbit_P_e}
A dynamical system $(X,f)$ exhibits property $P_e$ if and only if it has $\gamma$-restricted two-sided orbital limit shadowing.
\end{theorem}
\begin{proof}
First suppose that $(X,f)$ has property $P_e$. Let $\langle x_i \rangle_{i \in \mathbb{Z}}$ be a two-sided asymptotic pseudo-orbit such that $\alpha( \langle x_i \rangle )=\omega (\langle x_i \rangle )$. Let $A = \alpha( \langle x_i \rangle )=\omega (\langle x_i \rangle )$. By Lemma \ref{LemmaOmegaAlphaAsymPseudoAreICT}, $A \in ICT_f$. Therefore, since $(X,f)$ exhibits property $P_e$, there exists a full trajectory $\langle z_i \rangle _{i \in \mathbb{Z}}$ such that $\alpha( \langle z_i \rangle )=\omega (\langle z_i \rangle )=A$. I.e.\ $\alpha (\langle z_i \rangle)=\alpha (\langle x_i \rangle )$ and $\omega (\langle z_i \rangle)=\omega (\langle x_i \rangle)$. Therefore $(X,f)$ has $\gamma$-restricted two-sided orbital limit shadowing.

Now suppose that $(X,f)$ has $\gamma$-restricted two-sided orbital limit shadowing. Let $A \in ICT_f$ be given. By Lemma \ref{Lemma_KEYLEMMA_ICT}, there exists a two-sided asymptotic pseudo-orbit $\langle a_i \rangle_{i \in \mathbb{Z}}$ in $A$ such that $\alpha(\langle a_i \rangle) =A = \omega(\langle a_i \rangle)$. Let $\langle z_i \rangle_{i \in \mathbb{Z}}$ be a full trajectory which $\gamma$-restricted two-sided orbital limit shadows $\langle a_i \rangle_{i \in \mathbb{Z}}$. Then $\alpha (\langle z_i \rangle)=\alpha (\langle a_i \rangle )=A$ and $\omega (\langle z_i \rangle)=\omega (\langle a_i \rangle)=A$. In particular, $\alpha( \langle z_i \rangle )=\omega (\langle z_i \rangle )=A$. Therefore $(X,f)$ has property $P_e$.
\end{proof}

\begin{corollary}\label{corollary_sufficient_P_e}
If a system exhibits any of the following shadowing properties then it has property $P_e$:
\begin{enumerate}
    \item two-sided limit shadowing;
    \item two-sided orbital limit shadowing;
    \item $\delta$-restricted two-sided orbital limit shadowing.
\end{enumerate}
\end{corollary}
\begin{proof}
It suffices to note that each property implies $\gamma$-restricted two-sided orbital limit shadowing.
\end{proof}

\begin{corollary}
If $(X,f)$ is an expansive system with shadowing then it has $\gamma$-restricted two-sided orbital limit shadowing.
\end{corollary}
\begin{proof}
Shadowing and expansivity together give that $P_e$ is satisfied (Theorem \ref{thm_shad_plus_expansive_implies_P_e}). The result now follows from Theorem \ref{thm_char_full_orbit_P_e}.
\end{proof}

\subsection{$\gamma$-limit sets and $ICT_f$}\label{SectionGamma_Implications}
Before turning our attention to $P_a$, and in keeping with the precedent set in \cite{GoodMeddaughMitchell}, we will briefly highlight a corollary concerning $\gamma$-limit sets. These were introduced by Hero \cite{Hero} who studied them for interval maps, $\gamma$-limit sets have since been further examined by Sun \emph{et al.}\ in \cite{Sun} and \cite{Sun2} for graph maps and dendrites respectively. The $\gamma$-\emph{limit set} of a point $x$, denoted $\gamma(x)$, is defined by saying that, for any $y \in X$, $y \in \gamma (x)$ if and only if $y \in \omega(x)$ and there exists a sequence $\langle y_i\rangle_{i=1} ^\infty$ in $X$ and a strictly increasing sequence $\langle n_i\rangle_{i=1} ^\infty$ in $\mathbb{N}$ such that $f^{n_i}(y_i)=x$ for each $i$ and $\lim_{i \to \infty} y_i = y$. We denote by $\gamma_f$ the set of all $\gamma$-limit sets of $(X,f)$.

We gave the following remark in \cite{GoodMeddaughMitchell}.

\begin{remark}\label{RemarkGammaLimitSetsForHomeomorphisms}
For a dynamical system $(X,f)$ with $f$ a homeomorphism, for any $x \in X$, $\gamma(x)= \alpha(x) \cap \omega(x)$. (Recall that as $f$ is a homeomorphism we have defined $\alpha(x) =\alpha(\langle x_i \rangle _{i \leq 0})$, where $\langle x_i \rangle _{i \leq 0}$ is the unique backward trajectory of $x$.)
\end{remark}

Note that in contrast to $\omega$-limit sets in a compact setting, it is possible that a $\gamma$-limit set may be empty. Consider, for example, $f \colon [0,1] \to [0,1] \colon x \mapsto x^2$. Then, using the content of Remark \ref{RemarkGammaLimitSetsForHomeomorphisms}, we see that $\gamma(\frac{1}{2})= \emptyset$.

\begin{corollary}
If $(X,f)$ is a system with $\gamma$-restricted two-sided orbital limit shadowing then $ICT_f \subseteq \gamma_f$.
\end{corollary}
\begin{proof}
Let $A \in ICT_f$. By Theorem \ref{thm_char_full_orbit_P_e}, $(X,f)$ exhibits $P_e$: let $\langle x_i \rangle_{i \in \mathbb{Z}}$ be a full trajectory such that $\alpha(\langle x_i \rangle)=\omega(\langle x_i \rangle)=A$. Let $x = x_0$. Then $\omega(x)=A$. Since $\gamma(x) \subseteq \omega(x)$ by definition it follows that $\gamma(x) \subseteq A$. Furthermore, since $\langle x_i \rangle_{i \leq 0}$ is a backward trajectory from $x$ and $\alpha(\langle x_i \rangle)=A \subseteq \omega(x)$, it follows that $\gamma(x) \supseteq A$. Thus $\gamma(x)=A$. Since $A \in ICT_f$ was picked arbitrarily it follows that $ICT_f \subseteq \gamma_f$.
\end{proof}

\subsection{Property $P_a$} We now turn our attention to Question \ref{Q2} and property $P_a$, i.e.\ `for any $A \in ICT_f$ and any $\epsilon>0$ there exists a full trajectory $\langle x_i \rangle_{i \in \mathbb{Z}}$ such that $d_H(\alpha(\langle x_i \rangle, A) < \epsilon$ and $d_H( \omega(\langle x_i \rangle),A)< \epsilon$'.

Whilst for $P_a$ to hold it must be the case that $\overline{\alpha_f}=\overline{\omega_f}=ICT_f$, this alone is not sufficient. Example \ref{Example_SquareSystem} serves to demonstrate this.

In \cite{GoodMeddaugh2016} it is shown that the property of $\overline{\omega_f}=ICT_f$ is characterised by a variation on shadowing the authors term \emph{cofinal orbital shadowing}. A system $f \colon X \to X$ has the cofinal orbital shadowing property if for all $\epsilon>0$ there exists $\delta>0$ such that for any $\delta$-pseudo-orbit $\langle x_i \rangle _{i = 0} ^\infty$ there exists a point $ z \in X$ such that for any $K \in \mathbb{N}$ there exists $N \geq K$ such that 
\[d_H(\overline{\{f^{N+i}(z)\}_{i =0} ^\infty} , \overline{\{x_{N+i}\}_{i = 0} ^\infty})<\epsilon.\] 
In \cite{GoodMeddaughMitchell} we show that the natural backward analogue of this property, which we name \emph{backward cofinal orbital shadowing}, characterises when $\overline{\alpha_f}=ICT_f$. These notions motivate the following.

\begin{definition}
A system $(X,f)$ has two-sided cofinal orbital shadowing if for all $\epsilon>0$ there exists $\delta>0$ such that for any two-sided $\delta$-pseudo-orbit $\langle x_i \rangle _{i \in \mathbb{Z}}$ there exists a full trajectory $z \in X$ such that for any $K \in \mathbb{N}$ there exists $N \geq K$ such that 
\begin{enumerate}
    \item $d_H(\overline{\{z_{N+i}\}_{i\geq 0}}, \overline{\{x_{N+i}\}_{i \geq 0}})<\epsilon$;
    \item $d_H(\overline{\{z_{i-N}\}_{i \leq 0}} , \overline{\{x_{i-N}\}_{i \leq 0} })<\epsilon.$
\end{enumerate}
\end{definition}

We will see that this is indeed sufficient for $P_a$, however it is not necessary.

\begin{example}\label{Example_Torus_system}
Let $X= \mathbb{R}/\mathbb{Z} \times \mathbb{R}/\mathbb{Z}$ and let $f$ be given by $(x,y) \mapsto (x+\alpha, y)$, where $\alpha$ is some fixed irrational number. Then $(X,f)$ satisfies $P_a$ but does not have two-sided cofinal orbital shadowing.
\end{example}
Whilst we omit the construction of such a pseudo-orbit, it can be seen that in Example \ref{Example_Torus_system}, given any $\delta>0$, one can make a two-sided $\delta$-pseudo-orbit whose $\alpha$-limit set is $\mathbb{R}/\mathbb{Z} \times \{0\}$ and whose $\omega$-limit set is $\mathbb{R}/\mathbb{Z} \times \{\frac{1}{2}\}$. There is no full trajectory which satisfies the conditions in two-sided cofinal orbital shadowing for such a pseudo-orbit with $\epsilon< \frac{1}{4}$.

As in our search to characterise $P_e$, a restriction is necessary.

\begin{definition}
A system $(X,f)$ has $\gamma$-restricted two-sided cofinal orbital shadowing if for all $\epsilon>0$ there exists $\delta>0$ such that for any two-sided $\delta$-pseudo-orbit $\langle x_i \rangle _{i \in \mathbb{Z}}$ such that $d_H( \alpha(\langle x_i \rangle), \omega(\langle x_i \rangle))<\epsilon$ there exists a full trajectory $z \in X$ such that for any $K \in \mathbb{N}$ there exists $N \geq K$ such that 
\begin{enumerate}
    \item $d_H(\overline{\{z_{N+i}\}_{i \geq 0}}, \overline{\{x_{N+i}\}_{i \geq 0} })<\epsilon$;
    \item $d_H(\overline{\{z_{i-N}\}_{i \leq 0}} , \overline{\{x_{i-N}\}_{i \leq 0} })<\epsilon.$
\end{enumerate}
\end{definition}

\begin{remark}\label{remark_Equiv_twosided_cofinal}
It is equivalent to replace conditions (1) and (2) in the definition of $\gamma$-restricted two-sided cofinal orbital shadowing with the following:
\begin{enumerate}
    \item $d_H(\omega(\langle z_i \rangle), \omega(\langle x_i \rangle ))<\epsilon$;
    \item $d_H(\alpha(\langle z_i \rangle), \alpha(\langle x_i \rangle ))<\epsilon.$
\end{enumerate}
\end{remark}

\begin{theorem}\label{thm_char_approx_with_full_orbit_P_a}
A dynamical system $(X,f)$ exhibits property $P_a$ if and only if it has $\gamma$-restricted two-sided cofinal orbital shadowing.
\end{theorem}

\begin{proof}
First suppose that $(X,f)$ has property $P_a$. (We will use the content of Remark \ref{remark_Equiv_twosided_cofinal} to show $(X,f)$ has $\gamma$-restricted two-sided cofinal orbital shadowing.) Let $\epsilon>0$ be given. Now take $\eta= \frac{\epsilon}{3}$ and let $\delta= \eta$. Let $\langle x_i \rangle _{i \in \mathbb{Z}}$ be a two-sided $\delta$-pseudo-orbit such that $d_H( \alpha(\langle x_i \rangle), \omega(\langle x_i \rangle))<\eta$. Let $A = \alpha( \langle x_i \rangle )$ and $B=\omega (\langle x_i \rangle )$. By Lemma \ref{LemmaOmegaAlphaAsymPseudoAreICT}, $A, B \in ICT_f$. Therefore, since $(X,f)$ exhibits property $P_a$, there exists a full trajectory $\langle z_i \rangle _{i \in \mathbb{Z}}$ such that $d_H(\alpha( \langle z_i \rangle ), A)< \eta$ and $d_H(\omega( \langle z_i \rangle ), A)< \eta$. It follows by the triangle inequality that $d_H(\omega( \langle z_i \rangle ), B)< 2\eta < \epsilon$. Therefore $(X,f)$ has $\gamma$-restricted two-sided cofinal orbital shadowing.


Now suppose that $(X,f)$ has $\gamma$-restricted two-sided cofinal orbital shadowing. Let $A \in ICT_f$ be given and let $\epsilon>0$ be given. Take $\delta>0$ corresponding to $\epsilon$ for the formulation of $\gamma$-restricted two-sided orbital limit shadowing given by Remark \ref{remark_Equiv_twosided_cofinal} (without loss of generality $\delta < \epsilon$). By Lemma \ref{Lemma_KEYLEMMA_ICT}, there exists a two-sided asymptotic $\delta$-pseudo-orbit $\langle a_i \rangle_{i \in \mathbb{Z}}$ in $A$ such that $\alpha(\langle a_i \rangle) =A = \omega(\langle a_i \rangle)$. Let $\langle z_i \rangle_{i \in \mathbb{Z}}$ be a full trajectory for which
\begin{enumerate}
    \item $d_H(\omega(\langle z_i \rangle), \omega(\langle a_i \rangle ))<\epsilon$;
    \item $d_H(\alpha(\langle z_i \rangle), \alpha(\langle a_i \rangle ))<\epsilon.$
\end{enumerate}
In particular, $d_H(\alpha (\langle z_i \rangle), A) < \epsilon$ and $d_H(\omega (\langle z_i \rangle), A) < \epsilon$. Therefore $(X,f)$ satisfies property $P_a$.
\end{proof}

\begin{corollary}
If $(X,f)$ has shadowing then it has $\gamma$-restricted two-sided cofinal orbital shadowing.
\end{corollary}
\begin{proof}
By Theorem \ref{thm_shad_implies_P_a} a system with shadowing satisfies property $P_a$. The result follows from Theorem \ref{thm_char_approx_with_full_orbit_P_a}.
\end{proof}

\begin{corollary}
If a system $(X,f)$ has two-sided cofinal orbital shadowing then it has property $P_a$.
\end{corollary}

\begin{corollary}
If $(X,f)$ has $\gamma$-restricted two-sided orbital limit shadowing then it has $\gamma$-restricted two-sided cofinal orbital shadowing.
\end{corollary}
\begin{proof}
Since $P_e \implies P_a$, the result follows from theorems \ref{thm_char_full_orbit_P_e} and \ref{thm_char_approx_with_full_orbit_P_a}.
\end{proof}

\section{Closing examples}\label{Section_closing_examples}
We wrap up this paper by constructing two further examples. We start, in Example \ref{Example_periodic_cofinal}, by constructing a homeomorphism which exhibits $P_a$ but not $P_e$, thereby demonstrating that $\gamma$-restricted two-sided cofinal orbital shadowing does not imply $\gamma$-restricted two-sided orbital limit shadowing. We then close the paper by giving one final example (Example \ref{Example_Seq_of_squares}) which draws some of the themes in \cite{GoodMeddaugh2016}, \cite{GoodMeddaughMitchell} and the present paper, together.

\begin{example}\label{Example_periodic_cofinal}
Start with $Q$ as in Example \ref{Example_SquareSystem} and let $f$ act on these points in the same manner. Now, for each $n \in \mathbb{N}$, insert the set $R_n$. However, in contract to Example \ref{Example_SquareSystem}, let $f$ act on $R_n$ in a simple anticlockwise manner; so that each $R_n$ consists of a periodic orbit going anticlockwise. 
Then $(X,f)$ is a homeomorphism which satisfies property $P_a$ but not property $P_e$.
\end{example}
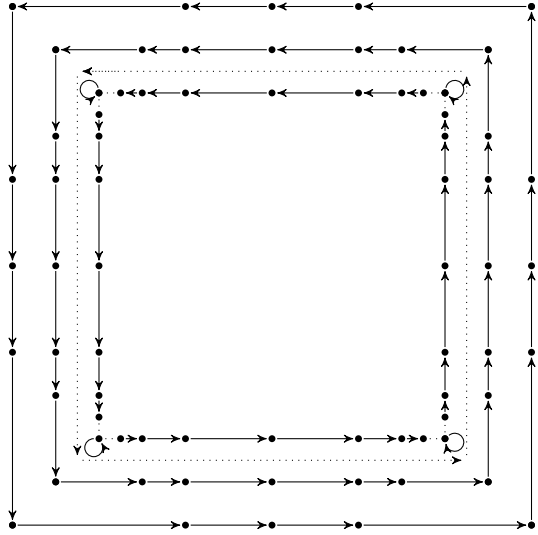
\begin{figure}[h!] \label{FigurePeriodicCofinalExample}
\centering
\begin{tikzpicture}[>=stealth',scale=2.3]

	\fill[black] (-1,-1) circle (.2mm);
	\node[circle,draw=white] (origin) at (-1,-1) {};
	\drawloop[->,stretch=0.4]{origin}{180}{300};

	\fill[black] (1,1) circle (.2mm);
	\node[circle,draw=white] (end) at (1,1) {};
	\drawloop[<-,stretch=0.4]{end}{320}{440};;

	\fill[black] (-1,1) circle (.2mm);
	\node[circle,draw=white] (origin) at (-1,1) {};
	\drawloop[->,stretch=0.4]{origin}{100}{220};
	
	\fill[black] (1,-1) circle (.2mm);
	\node[circle,draw=white] (end) at (1,-1) {};
	\drawloop[<-,stretch=0.4]{end}{280}{400};;
	
	\foreach \y in {0,0.5,-.5,0.75,-.75,0.875,-.875,}
		{
		\fill[black] (-1,\y) circle (.2mm);
		}
	\foreach \y/\x in {0.845/0.78,0.72/.53,.47/.03,-.03/-.47,-.53/-.72,-.78/-.845}
		{
		\draw[->] (-1,\y) -- (-1,\x);
		}
		\draw[dotted, -] (-1,1) -- (-1,0.875);
		\draw[dotted, -] (-1,-.875) -- (-1,-1);

		\foreach \y in {0,0.5,-.5,0.75,-.75,0.875,-.875,}
		{
		\fill[black] (1,\y) circle (.2mm);
		}
		\foreach \y/\x in {0.845/0.78,0.72/.53,.47/.03,-.03/-.47,-.53/-.72,-.78/-.845}
		{
		\draw[->] (1,\x) -- (1,\y);
		}
		\draw[dotted, -] (1,0.875) -- (1,1);
		\draw[dotted, -] (1,-1) -- (1,-.875);

		\foreach \x in {0,0.5,-.5,0.75,-.75,0.875,-.875,}
		{
		\fill[black] (\x,-1) circle (.2mm);
		}
		\foreach \y/\x in {0.845/0.78,0.72/.53,.47/.03,-.03/-.47,-.53/-.72,-.78/-.845}
		{
		\draw[->] (\x,-1) -- (\y,-1);
		}
		\draw[dotted, -] (0.875,-1) -- (1,-1);
		\draw[dotted, -] (-1,-1) -- (-.875,-1);

		\foreach \x in {0,0.5,-.5,0.75,-.75,0.875,-.875,}
		{
		\fill[black] (\x,1) circle (.2mm);
		}
			\foreach \y/\x in {0.845/0.78,0.72/.53,.47/.03,-.03/-.47,-.53/-.72,-.78/-.845}
		{
		\draw[->] (\y,1) -- (\x,1);
		}
		\draw[dotted, -] (1,1) -- (0.875,1);
		\draw[dotted, -] (-.875,1) -- (-1,1);

		\foreach \x in {0.5,-.5,0,1.5,-1.5}
		{
		\fill[black] (\x,1.5) circle (.2mm);
		}
		\draw[->] (-0.03,1.5) -- (-0.47,1.5);
		\draw[->] (0.47,1.5) -- (0.03,1.5);
		\draw[->] (1.47,1.5) -- (0.53,1.5);
		\draw[->] (-0.53,1.5) -- (-1.47,1.5);
		
		\foreach \x in {0.5,-.5,0,1.5,-1.5}
		{
		\fill[black] (\x,-1.5) circle (.2mm);
		}
		\draw[->]  (-1.47,-1.5) -- (-0.53,-1.5);
		\draw[->]  (-.47,-1.5) -- (-0.03,-1.5);	
		\draw[->]  (.03,-1.5) -- (.47,-1.5);
		\draw[->]  (.53,-1.5) -- (1.47,-1.5);
		
		\foreach \y in {0.5,-.5,0,1.5,-1.5}
		{
		\fill[black] (1.5,\y) circle (.2mm);
		}
		\draw[->]  (1.5,-1.47) -- (1.5,-0.53);
		\draw[->]  (1.5,-.47) -- (1.5,-0.03);	
		\draw[->]  (1.5,.03) -- (1.5,.47);
		\draw[->]  (1.5,.53) -- (1.5,1.47);
		
		\foreach \y in {0.5,-.5,0,1.5,-1.5}
		{
		\fill[black] (-1.5,\y) circle (.2mm);
		}
		\draw[->] (-1.5,1.47) -- (-1.5,0.53);
		\draw[->] (-1.5,0.47) -- (-1.5,0.03);
		\draw[->] (-1.5,-0.03) -- (-1.5,-0.47);
		\draw[->] (-1.5,-0.53) -- (-1.5,-1.47);
		
		\foreach \x in {0.5,-.5,0,-.75,.75, 1.25, -1.25}
		{
		\fill[black] (\x,1.25) circle (.2mm);
		}
		\draw[->] (-0.53,1.25) -- (-.72,1.25);
		\draw[->] (1.22,1.25) -- (.78,1.25);
		\draw[->] (.72,1.25) -- (0.53,1.25);
	    \draw[->] (.47,1.25) -- (0.03,1.25);
		\draw[->] (-0.03,1.25) -- (-.47,1.25);
        \draw[->] (-0.78,1.25) -- (-1.22,1.25);

		\foreach \x in {0.5,-.5,0,-.75,.75, 1.25, -1.25}
		{
		\fill[black] (\x,-1.25) circle (.2mm);
		}
			\draw[->] (-1.22,-1.25) -- (-.78,-1.25);
		 \draw[->] (-.72,-1.25) -- (-0.53,-1.25);
		 	\draw[->] (-.47,-1.25) -- (-0.03,-1.25);
		 	\draw[->]  (.03,-1.25) -- (.47,-1.25);
		 	 \draw[->] (0.53,-1.25) -- (.72,-1.25);
		 	 \draw[->] (0.78,-1.25) -- (1.22,-1.25);
		
		\foreach \y in {0.5,-.5,0,-.75,.75, 1.25, -1.25}
		{
		\fill[black] (1.25,\y) circle (.2mm);
		}
		 \draw[->] (1.25,-1.22) -- (1.25,-0.78);
		 \draw[->] (1.25,-.72) -- (1.25,-0.53);
		 	\draw[->] (1.25,-.47) -- (1.25,-0.03);
		 	\draw[->]  (1.25,.03) -- (1.25,.47);
		 	 \draw[->] (1.25,0.53) -- (1.25,.72);
		 	  \draw[->] (1.25,.78) -- (1.25,1.22);
		
		\foreach \y in {0.5,-.5, 0,-.75,.75, 1.25, -1.25}
		{
		\fill[black] (-1.25,\y) circle (.2mm);
		}
		 \draw[->] (-1.25,1.22) -- (-1.25, .78);
			 \draw[->] (-1.25,.72) -- (-1.25, 0.53);
		 	\draw[->] (-1.25,.47) -- (-1.25,0.03);
		 	\draw[->]  (-1.25, -.03) -- (-1.25,-.47);
		 	 \draw[->] (-1.25, -0.53) -- (-1.25,-.72);
		 	 \draw[->] (-1.25,-.78) -- (-1.25, -1.22);

		 		\draw[dotted, ->] (1.095,1.125) -- (-1.095,1.125);
		 		\draw[dotted, ->] (1.125,-1.095) -- (1.125,1.095);
		 			\draw[dotted, ->] (-1.095,-1.125) -- (1.095,-1.125);
		 			\draw[dotted, ->] (-1.125,1.095) -- (-1.125,-1.095);
		 			\draw[dotted, ->] (-0.905,1.125) -- (-1.095,1.125);

\end{tikzpicture}
\caption{The system in Example \ref{Example_periodic_cofinal}}
\end{figure}
It is not difficult to see that $\alpha_f$ and $\omega_f$ are both equal to $\{R_n \mid n \in \mathbb{N}\}\cup \{ \{(1,1)\},\{(-1,1)\},\{(-1,-1)\}, \{(1,-1)\}\}$. Meanwhile $ICT_f$ additionally includes $Q$. Because, for instance, $ICT_f \neq \alpha_f$, it follows that $P_e$ is not satisfied by $(X,f)$. However, for any $\epsilon>0$ there is a full trajectory whose $\alpha$-limit set and $\omega$-limit set both lie within $\epsilon$ of $Q$ (resp.\ $2Q$). To see this observe that the subsystem $(Q, f\restriction_{Q}))$ is the limit of the subsystems $(R_n, f\restriction_{R_n})$. Let $\epsilon>0$ be given and let $n \in \mathbb{N}$ be such that $d_H(R_n,Q)< \epsilon$. Pick $z \in R_n$ and let $\langle z_i \rangle_{i \in \mathbb{Z}}$ be the unique full trajectory with $z_0=z$; this is a periodic orbit with $\alpha(\langle z_i \rangle)=\omega(\langle z_i \rangle)=R_n$. Therefore $d_H(\alpha(\langle z_i \rangle), Q) < \epsilon$ and $d_H(\omega(\langle z_i \rangle), Q) < \epsilon$. Hence $P_a$ holds and, in particular, $\overline{\alpha_f}=\overline{\omega_f}=ICT_f$.

\medskip

As stated at the beginning of this section, our final example (Example \ref{Example_Seq_of_squares}) serves to draw some of the themes in \cite{GoodMeddaugh2016}, \cite{GoodMeddaughMitchell} and in the present paper, together. Example \ref{Example_Seq_of_squares} is an informal construction of a homeomorphism which exhibits neither $P_e$ nor $P_a$ and for which 
\begin{enumerate}
    \item $\alpha_f=\omega_f \neq ICT_f$; and,
    \item $\overline{\alpha_f}=\overline{\omega_f} = ICT_f$.
\end{enumerate}
Furthermore, the only non-singleton elements of $ICT_f$ which may be approximated by the $\alpha$-limit set and $\omega$-limit set of the same full trajectory are precisely the ones which belong to neither $\alpha_f$ nor $\omega_f$. (Hence the system does not have shadowing.)

Before we give this last example, we recall that the system in Example \ref{Example_SquareSystem} has the forward and backward versions of both cofinal orbital shadowing and orbital limit shadowing. This is because $\alpha_f= \omega_f= ICT_f$. The system in Example \ref{Example_Seq_of_squares}, on the other hand, will have the forward and backward versions of cofinal orbital shadowing (since $\overline{\alpha_f}=\overline{\omega_f} = ICT_f$), but neither the forward nor backward version of orbital limit shadowing (since $\alpha_f=\omega_f \neq ICT_f$).

\begin{example}\label{Example_Seq_of_squares}
Start with $Q$ as in Example \ref{Example_SquareSystem} and 
let $f$ act on these points in the same manner. For $n \in \mathbb{R}^+$, define the set $nQ \coloneqq \{(nx,ny) \mid (x,y) \in Q\}$. For each $n \in \mathbb{N}$ insert the sets $\frac{2^{n+1}-1}{2^n}Q$ and $\frac{2^n+1}{2^n}Q$. Also insert the set $2Q$. Let $f$ act on these akin to the way it acts on $Q$. 

Now, for each $n \in \mathbb{N}$, insert a two-sided sequence of points which lies between $\frac{2^{n+1}-1}{2^n}Q$ and $\frac{2^{(n+1)+1}-1}{2^{n+1}}Q$ in the plane such that 
\begin{enumerate}
    \item each point maps onto the next in the sequence;
    \item the $\omega$-limit set of every point in the sequence is $\frac{2^{(n+1)+1}-1}{2^{n+1}}Q$; and,
    \item the $\alpha$-limit set of every point in the sequence is $\frac{2^{n+1}-1}{2^{n}}Q$.
\end{enumerate}
(Combining, and making suitable adjustments to, some of the techniques used in Example \ref{Example_SquareSystem} would be one appropriate way to accomplish this.)

Finally, for each $n \in \mathbb{N}$ with $n \geq 2$, insert a two-sided sequence of points which lies between $\frac{2^n+1}{2^n}Q$ and $\frac{2^{n-1}+1}{2^{n-1}}Q$ in the plane such that 
\begin{enumerate}
    \item each point maps onto the next in the sequence;
    \item the $\omega$-limit set of every point in the sequence is $\frac{2^{n-1}+1}{2^{n-1}}Q$; and,
    \item the $\alpha$-limit set of every point in the sequence is $\frac{2^n+1}{2^n}Q$.
\end{enumerate}
\end{example}
Whilst we omit a proof of the fact, it is not difficult to see that $\alpha_f$ and $\omega_f$ are comprised of $\frac{2^{n+1}-1}{2^n}Q$ and $\frac{2^n+1}{2^n}Q$ (for each $n \in \mathbb{N}$), along with the singleton sets of all fixed points in the system. Meanwhile $ICT_f$ additionally includes $Q$ and $2Q$. This implies that $P_e$ does not hold. Note further that no set of the form $\frac{2^{n+1}-1}{2^n}Q$ may be approximated to any given accuracy by \emph{both} the $\alpha$-limit set and $\omega$-limit set of the same full trajectory. However, for any $\epsilon>0$ there is a full trajectory whose $\alpha$-limit set and $\omega$-limit set both lie within $\epsilon$ of $Q$ (resp.\ $2Q$). To see this, observe that the subsystem $(Q, f\restriction_{Q}))$ is the limit of the sequence of subsystems $(\frac{2^n+1}{2^n}Q, f\restriction_{\frac{2^n+1}{2^n}Q})$ as $n \to \infty$. Let $\epsilon>0$ be given and let $n \in \mathbb{N}$ be such that $d_H(\frac{2^n+1}{2^n}Q,Q)< \epsilon$. Let $\langle z_i \rangle_{i \in \mathbb{Z}}$ be a full trajectory for which $\omega(\langle z_i \rangle) =\frac{2^{n-1}+1}{2^{n-1}}Q$ and $\alpha(\langle z_i \rangle)=\frac{2^n+1}{2^n}Q$. Then $d_H(\alpha(\langle z_i \rangle), Q) < \epsilon$ and $d_H(\omega(\langle z_i \rangle), Q)< \epsilon$. By noting that the subsystem $(2Q, f\restriction_{2Q}))$ is the limit of the sequence of subsystems $(\frac{2^{n+1}-1}{2^n}Q, f\restriction_{\frac{2^{n+1}-1}{2^n}Q})$ as $n \to \infty$, a similar argument may be given with regard to $2Q$. In particular it follows that $\overline{\alpha_f}=\overline{\omega_f}=ICT_f$.




\bibliographystyle{plain} 
\bibliography{bib}

\end{document}